\documentclass[12pt]{amsproc}

\usepackage{amssymb}

\theoremstyle{plain}
\newtheorem{thm}{Theorem}[section]
\newtheorem{prop}[thm]{Proposition} 
\newtheorem{lemma}[thm]{Lemma}
\newtheorem{cor}[thm]{Corollary}
\newtheorem{con}[thm]{Conjecture}
\newtheorem*{conjec}{Strong Conciseness Conjecture}

\theoremstyle{definition}
\newtheorem{defn}[thm]{Definition}
 
\newtheorem{example}[thm]{Example}

\newtheorem*{acknowledgements}{Acknowlwdgements}

\theoremstyle{remark}
\newtheorem{rem}[thm]{Remark}

\numberwithin{equation}{section}
\newcommand{\uo}{\underline{\omega}}
\renewcommand{\epsilon}{\varepsilon}
\newcommand{\ue}{\underline{\epsilon}}
\newcommand{\kr}{r} % easy to change the length of the word 
\DeclareMathOperator*{\dotprod}{\dot\prod}
\title[Strong conciseness in profinite groups]{Strong conciseness in
  profinite groups}

\author[E.\ Detomi]{Eloisa Detomi}

\address{Dipartimento di Ingegneria dell'Informazione\\
 Universit\`a  degli Studi di Padova \\
Via Gradenigo 6/b\\
 35131 Padova, Italy} 
\email{eloisa.detomi@unipd.it}

\author[B.\ Klopsch]{Benjamin Klopsch}
\address{Heinrich-Heine-Universit\"{a}t D\"{u}sseldorf\\
 Mathematisches Institut\\
 Universit\"{a}tsstr. 1\\
 40225\\
 D\"{u}sseldorf, Germany} 
\email{klopsch@math.uni-duesseldorf.de}

\author[P.\ Shumyatsky]{Pavel Shumyatsky}
\address{Department of Mathematics\\
 University of Brasilia\\
 Brasilia-DF
 70910-900 \\
 Brazil} 
\email{pavel@unb.br}

\begin{document}
\maketitle

\begin{abstract}
  A group word $w$ is said to be strongly concise in a class
  $\mathcal{C}$ of profinite groups if, for every group $G$ in
  $\mathcal{C}$ such that $w$ takes less than $2^{\aleph_0}$ values
  in~$G$, the verbal subgroup $w(G)$ is finite.  Detomi, Morigi and
  Shumyatsky established that multilinear commutator words -- and the
  particular words $x^2$ and~$[x^2,y]$ -- have the property that the
  corresponding verbal subgroup is finite in a profinite group $G$
  whenever the word takes at most countably many values in $G$.  They
  conjectured that, in fact, this should be true for every word.  In
  particular, their conjecture included as open cases power words and
  Engel words.

  In the present paper, we take a new approach via parametrised words
  that leads to stronger results.  First we prove that multilinear
  commutator words are strongly concise in the class of all profinite
  groups.  Then we establish that every group word is strongly concise
  in the class of nilpotent profinite groups.  From this we deduce,
  for instance, that, if $w$ is one of the group words $x^2$, $x^3$,
  $x^6$, $[x^3,y]$ or $[x,y,y]$, then $w$ is strongly concise in the
  class of all profinite groups.  Indeed, the same conclusion can be
  reached for all words of the infinite families
  $[x^m,z_1,\ldots,z_r]$ and $[x,y,y,z_1,\ldots,z_r]$, where
  $m \in \{2,3\}$ and~$r \ge 1$.
 \end{abstract}

%%%%%

\section{Introduction} 
Let $w = w(x_1,\ldots,x_{\kr})$ be a group word, i.e.\ an element of the
free group on $x_1, \ldots, x_{\kr}$.  We take an interest in the set of
all $w$-values in a group $G$ and the verbal subgroup generated by it;
they are
\[
G_w = \{ w(g_1, \ldots, g_{\kr}) \mid g_1, \ldots, g_{\kr} \in G \} \qquad
\text{and} \qquad w(G) = \langle G_w \rangle.
\] 
In the context of topological groups $G$, we write $w(G)$ to denote
the closed subgroup generated by all $w$-values in~$G$.

The word $w$ is said to be \emph{concise} in a class $\mathcal{C}$ of
groups if, for each $G$ in $\mathcal{C}$ such that $G_w$ is finite, also
$w(G)$ is finite.  For topological groups, especially profinite
groups, a variation of the classical notion arises quite naturally: we
say that $w$ is \emph{strongly concise} in a class $\mathcal{C}$ of
topological groups if, for each $G$ in $\mathcal{C}$, already the
bound $\lvert G_w \rvert < 2^{\aleph_0}$ implies that $w(G)$ is
finite.

A conjecture proposed by Philip Hall (e.g.\ see~\cite{Tu64}) predicted
that every word $w$ would be concise in the class of all groups, but
almost three decades later the assertion was famously refuted by
Ivanov~\cite{Iv89}.  On the other hand, Merzlyakov~\cite{Me67} showed
already in the 1960s that every word is concise in the class of linear
groups.  This naturally leads to the question whether every word is
concise in the class of residually finite groups, or equivalently in
the class of profinite groups.  Lately, this question was highlighted
by Jaikin-Zapirain~\cite{Ja08}, who used Merzlyakov's theorem in his
investigations of verbal width in finitely generated pro-$p$
groups; compare also~\cite{Se09}.

In~\cite{DeMoSh16}, Detomi, Morigi and Shumyatsky suggested a
strengthened profinite version of Hall's conciseness
conjecture, namely that for every  word $w$ and
 every profinite group $G$, the bound
$\lvert G_w \rvert \le {\aleph_0}$ implies that $w(G)$ is finite.
They verified this for multilinear commutator words, also known as
 outer-commutator  words (see Section~\ref{sec:mcw}), as
well as  for  the particular words $x^2$ and~$[x^2,y]$.
Their considerations relied on the Baire category theorem, but a more
direct argument (see Section~\ref{sec:preliminaries}) allows us to
deal with a  natural  stronger form of the conjecture. 

 For short, we say that a word $w$ is strongly concise if
it is strongly concise in the class of all profinite groups.

\begin{conjec}
  Every group word $w$ is strongly concise.
\end{conjec}

In the present paper, we initiate a systematic investigation of this
conjecture and produce positive evidence for it.  Among the words
treated in~\cite{DeMoSh16}, the special power word $x^2$ is the only
one for which a simple replacement of the Baire category theorem by
Proposition~\ref{pro:instead-of-Baire} below yields that it is
strongly concise.  More work is needed to confirm the Strong
Conciseness Conjecture for multilinear commutator words.

\begin{thm}\label{thm:mcw-more} Every multilinear
  commutator word is strongly concise.
\end{thm}

Guided by an interest in power words $x^m$ of exponent $m \ge 3$ and
$n$-Engel words $[x,_{n}y] = [x,y,\ldots,y]$, where $y$ appears
$n \ge 2$ times, we began an investigation of some specific words,
such as $x^3$ and~$[x,y,y]$.  Later we discovered that
 the  relevant computations could be subsumed under a
common approach.   The main outcome  of this consolidation
is the following result.

\begin{thm} \label{thm:w-str-concise-in-nilp-gp}
  Every group word $w$ is strongly concise in the class of nilpotent
  profinite groups.
\end{thm}

A straightforward and well-known argument shows that every group word
is strongly concise in the class of abelian profinite groups; compare
Proposition~\ref{pro:abelian-OK}.  But strong conciseness does not
behave well under group extensions;
Theorem~\ref{thm:w-str-concise-in-nilp-gp} and, more importantly, the
considerations that enter into its proof are new, even for nilpotent
groups of class~$2$.

  The following corollaries can be derived from
Theorem~\ref{thm:mcw-more} and
Theorem~\ref{thm:w-str-concise-in-nilp-gp} without further difficulty.

\begin{cor} \label{cor:F-mod-w(F)-nilp-OK} Let $F$ be a free group of
  countably infinite rank and let $w$ be a group word such that
  $F/w(F)$ is nilpotent.  Then $w$ is strongly concise.
\end{cor}

\begin{cor} \label{cor:specific-words} The following group words $w$
  are strongly concise:
  \[
  x^2, \quad x^3, \quad x^6, \quad [x,y,y] \qquad \text{and}
  \]
  \[
  [x^2,z_1, \ldots, z_r], \quad [x^3,z_1, \ldots, z_r], \quad
  [x,y,y,z_1, \ldots, z_r] \qquad \text{for $r \ge 1$,}
  \]
  where $x,y,z_1, z_2, \ldots$ are independent variables.
\end{cor}

Our proof of Theorem~\ref{thm:w-str-concise-in-nilp-gp} is based on
parametrised words; see Section~\ref{sec:para-words}.  Nilpotency is a
key ingredient for setting up induction parameters that help us to
reduce the complexity of the word $w$ as well as the complexity of the
group $G$ under consideration.

As a byproduct, our approach highlights the relevance of the following
two weaker versions of the Strong Conciseness Conjecture.

\begin{con} \label{con:w(G)-fin-gen} Suppose that the group word $w$
  has less than $2^{\aleph_0}$ values in a profinite group~$G$. Then
  $w(G)$ is generated by finitely many $w$-values.
\end{con}

\begin{con}
  Suppose that the group word $w$ has less than $2^{\aleph_0}$ values
  in a profinite group~$G$.  Then there is an open subgroup $H$ of $G$
  such that $w(H) = 1$.
\end{con}

To illustrate the relevance of Conjecture~\ref{con:w(G)-fin-gen}, we
summarise some conditional results that we obtained.  For this we
recall that if a group word $w$ `implies virtual nilpotency', then for
a large class of groups~$G$, including all finitely generated
residually finite groups, $w(G) = 1$ implies that $G$ is
nilpotent-by-finite, due to results of Burns and
Medvedev~\cite{BuMe03}.  Furthermore, following~\cite{GuSh15} we say
that a group word $w$ is `weakly rational' if for every finite group
$G$ and for every positive integer $e$ with
$\gcd(e,\lvert G \rvert) = 1$, the set $G_w$ is closed under taking
$e$th powers of its elements.  We refer to
Section~\ref{sec:fin-many-w-values} for a more detailed discussion of
these notions.

\begin{thm} \label{thm:vitual-nilp-and-weakly-rational} Let $w$ be a
   group word that \textup{(i)} implies virtual
  nilpotency or \textup{(ii)} is weakly rational.  Let $G$ be a
  profinite group such that $\lvert G_w \rvert < 2^{\aleph_0}$.  If
  $w(G)$ is generated by finitely many $w$-values, then $w(G)$ is
  finite.
\end{thm}

\smallskip

\noindent \emph{Notation and Organisation.}  Our notation is mostly
standard.  All repeated commutators are left-normed, e.g.\
$\gamma_3(x,y,z) = [x,y,z] = [[x,y],z]$.

In Section~\ref{sec:preliminaries} we collect some known results and
several basic observations; the elementary
Proposition~\ref{pro:instead-of-Baire} is one of the early key
insights.  In Section~\ref{sec:mcw} we prove that multilinear
commutator words are strongly concise.  The main results in
Section~\ref{sec:fin-many-w-values} are
Propositions~\ref{pro:delta-k-q-result}, \ref{pro:virt-nilp}
and~\ref{pro:weakly-rational}; in particular, the latter two yield
Theorem~\ref{thm:vitual-nilp-and-weakly-rational}.  In
Section~\ref{sec:para-words} we set up the reduction arguments based
on parametrised words. In Section~\ref{sec:nilp-groups} we prove
Theorem~\ref{thm:w-str-concise-in-nilp-gp} as well as
Corollaries~\ref{cor:F-mod-w(F)-nilp-OK} and~\ref{cor:specific-words}.

%%%%%

\section{Preliminaries} \label{sec:preliminaries}

In this section we collect some known results as well as several
straightforward consequences and basic observations.

For simplicity and to steer clear of the Continuum
Hypothesis (or Martin's Axiom), we record the following proposition
that helps us to avoid references to the Baire category theorem,
which appear frequently in~\cite{DeMoSh16} and related
  articles.

\begin{prop} \label{pro:instead-of-Baire} Let $\varphi \colon X \to Y$
  be a continuous map between non-empty profinite spaces that is
  nowhere locally constant, i.e.\ there exists no non-empty open
  subset $U \subseteq_\mathrm{o} X$ such that $\varphi \vert_U$ is
  constant.  Then 
  $\lvert X \varphi \rvert \ge 2^{\aleph_0}$.
\end{prop}

\begin{proof}
  For every non-empty closed open subset $U \subseteq X$ choose a
  continuous map $\vartheta_U \colon Y \to Z_U$ onto a finite discrete
  space $Z_U$ such that $\varphi \, \vartheta_U \colon X \to Z_U$ is
  not constant on~$U$.  Choose non-empty distinct fibers $U_1, U_2$ of
  the restriction of $\varphi \, \vartheta_U$ to~$U$; then
  $U_1, U_2 \subseteq U$ are non-empty closed open subsets of $X$ with
  $U_1 \varphi \cap U_2 \varphi = \varnothing$.

  Fix a non-empty closed open subset $A \subseteq X$,
    e.g.\ $A =X$.  For every sequence
    $ \mathbf{i} = (i_1,i_2, i_3,\ldots )$ in $\{1,2\}$, the
    consideration above yields a descending chain of non-empty closed
    open subsets
    $A_{i_1} \supseteq (A_{i_1})_{i_2}\supseteq
    ((A_{i_1})_{i_2})_{i_3}\supseteq \ldots$, and we set 
  \[
  A_\mathbf{i} = \bigcap\nolimits_{n \in \mathbb{N}} (
  \cdot\!\cdot\!\cdot ( (A_{i_1})_{i_2})_{i_3} \cdots )_{i_n}
  \subseteq_\mathrm{c} X.
  \]
  Since $X$ is compact, each $A_\mathbf{i}$ is non-empty, and
   we choose $a_\mathbf{i} \in A_\mathbf{i}$.  By
  construction we have $a_\mathbf{i} \varphi \ne a_\mathbf{j} \varphi$ for
  $\mathbf{i} \ne \mathbf{j}$.  Hence
  \[
  B = \big\{ a_\mathbf{i} \mid \mathbf{i} \in \{1,2\}^\mathbb{N}
  \big\} \subseteq X
  \]
  is mapped injectively into $Y$ under $\varphi$,
  and~$\lvert X \varphi \rvert \ge \lvert B \varphi \rvert = 2^{\aleph_0}$. 
\end{proof}

\begin{lemma} \label{lem:fin-conj-cl} Let $G$ be a profinite group and
  let $x \in G$.  If the conjugacy class $\{ x^g \mid g \in G \}$
  contains less than $2^{\aleph_0}$ elements, then it is finite.
\end{lemma}

\begin{proof}
   The set $\{ x^g \mid g \in G \}$ is in bijection with
    the coset space $G/\mathrm{C}_G(x)$, a homogeneous profinite
    space.  Alternatively, one can adapt the proof
    of~\cite[Lemma~3.1]{DeMoSh16}, using
    Proposition~\ref{pro:instead-of-Baire} in place of the Baire
    category theorem. 
\end{proof}

\begin{prop} \label{pro:abelian-OK} Every group word is strongly
  concise in the class of abelian profinite groups.
\end{prop}

\begin{proof}
  Let $G$ be an abelian profinite group.  It is enough to consider
  power words $w(x) = x^n$, where $n \in \mathbb{N}$.  For these we
  observe that $w(G) = \{ g^n \mid g \in G \} = G_w$, as $G \to G$,
  $g \mapsto g^n$ is a homomorphism.  Hence $w(G)=G_w$ is finite or
  has cardinality at least~$2^{\aleph_0}$.
\end{proof}

\begin{lemma} \label{lem:not-comm-periodic} Let $w \in F$ be an
  element of a free group $F$ such that $w \not \in [F,F]$.  Let $G$
  be a profinite group such that $\lvert G_w \rvert < 2^{\aleph_0}$.
  Then $G$ is periodic.
\end{lemma}

\begin{proof}
  Write $w(x_1,\ldots,x_{\kr}) = x_1^{\, e_1} \cdots x_{\kr}^{\, e_{\kr}} v$,
  where $e_1, \ldots, e_{\kr} \in \mathbb{Z}$ are not all zero and
  $v \in [F,F]$.  Then the word $y^m = w(y^{f_1},\ldots,y^{f_{\kr}})$,
  where
  $m = \sum_{i=1}^{\kr} e_i f_i = \gcd(e_1,\ldots,e_{\kr}) \in \mathbb{N}$,
  takes less than $2^{\aleph_0}$ values in~$G$.  By
  Proposition~\ref{pro:abelian-OK}, every procyclic subgroup of $G$ is
  finite, and thus $G$ is periodic.
\end{proof}

%%%%%
\section{Multilinear commutator words}\label{sec:mcw}

In this section we prove that every multilinear
  commutator word is strongly concise.  Recall that a multilinear
commutator word, also known as an outer-commutator word,
  is obtained  by nesting commutators and using each variable only
once.  Thus the word $[[x_1,x_2],[x_3,x_4,x_5],x_6]$ is a multilinear
commutator word  while the $3$-Engel word $[x,y,y,y]$ is
not.  An important family of multilinear commutator words
consists of  the repeated commutator words $\gamma_k$ 
on $k$ variables,  given by $\gamma_1=x_1$ and
$\gamma_k=[\gamma_{k-1},x_k]= [x_1,\ldots,x_k]$ for
  $k \ge 2$.  The verbal subgroup $\gamma_k(G)$ of a group $G$ is the
$k$th term of the lower central series of~$G$.  The
  derived words $\delta_k$, on $2^k$ variables, form another
  distinguished family of multilinear commutators; they are defined
  by $\delta_0=x_1$ and
$\delta_k = [\delta_{k-1}(x_1,\ldots,x_{2^{k-1}}),
\delta_{k-1}(x_{2^{k-1}+1}, \ldots,x_{2^k})].$
 The verbal subgroup  $\delta_k(G)=G^{(k)}$ is the
 $k$th  derived subgroup of $G$.

 Relying on the Baire category theorem, Detomi, Morigi and
 Shumyatsky~\cite{DeMoSh16} proved that, if $w$ is a multilinear
 commutator word, then for every profinite group $G$ the bound
 $\lvert G_w \rvert \le {\aleph_0}$ implies that $w(G)$ is finite.
 Proposition~\ref{pro:instead-of-Baire} enables us to strengthen this
 result: we show -- without recourse to the Continuum Hypothesis (or
 Martin's Axiom) -- that every multilinear commutator word is strongly
 concise.  For this we employ combinatorial techniques that were
 developed in~\cite{DMS-nilpotent,DMS-cosets} specifically for
 handling multilinear commutator words.

Throughout this section, we fix $\kr \in \mathbb{N}$ and a
multilinear commutator word
\[
w = w(x_1,\dots , x_{\kr}).
\]
Furthermore, $G$ is a profinite group.  For
$A_1, \dots , A_{\kr} \subseteq G$, we denote by
\[
w(A_1, \dots , A_{\kr})
\]
the subgroup generated by all $w$-values $w(a_1, \dots , a_{\kr})$,
where $a_i \in A_i$ for $1 \le i \le r$.  For
$I \subseteq \{1, \ldots, {\kr} \}$ we write
$\overline{I} = \{1, \dots ,{\kr} \} \smallsetminus I$.  For families
of variables $\mathbf{y} = (y_i)_{i \in I}$,
$\mathbf{z} = (z_i)_{i \in \overline{I}}$ we define
\[
w_I(\mathbf{y};\mathbf{z}) = w(u_1, \ldots, u_r), \quad
\text{where } u_s =
\begin{cases}
  y_s & \text{if $s \in I$,} \\
  z_s & \text{if $s \not\in I$.}
\end{cases}
\]
The notation extends to families $\mathbf{A} = (A_i)_{i \in I}$,
$\mathbf{B} = (B_i)_{i \in \overline{I}}$ of subsets of $G$ in the
natural way: $w_I(\mathbf{A};\mathbf{B})$ denotes the subgroup
generated by the relevant $w$-values.  For short, we write
$w_I(y_i;z_i)$ in place of $w_I(\mathbf{y};\mathbf{z})$ and
$w_I(A_i;B_i)$ in place of~$w_I(\mathbf{A};\mathbf{B})$.

The following are corollaries of \cite[Lemma~2.5]{DMS-cosets} and
\cite[Lemma~4.1]{DMS-nilpotent}.

\begin{cor}\label{cor:uno} Let $H \trianglelefteq_\mathrm{c} G$.
  Suppose that $g_1, \ldots, g_{\kr} \in G$ and $g \in G$ are such
  that $w(g_1 h_1, \ldots, g_{\kr} h_{\kr}) =  g $ for 
  all  $h_1, \ldots, h_{\kr} \in H$.
  Then
  $w_I(g_i H; H) = 1$ for every proper subset
  $I \subsetneqq \{1,\dots,{\kr}\}$.
\end{cor}

\begin{cor}\label{cor:M} 
  Let $H \trianglelefteq_\mathrm{c} G$, and suppose that
  $I \subseteq \{1, \dots ,{\kr} \}$ is such that $w_J(G; H) = 1$ for all
  $J \subsetneqq I$.  Then $w_I(g_ih_i; h_i) = w_I(g_i;h_i)$ for all $g_i \in G$, 
  $i \in I$, and all $h_1, \ldots, h_{\kr} \in H$.
\end{cor}

Next we employ the hypothesis $\lvert G_w \rvert < 2^{\aleph_0}$.

\begin{lemma}\label{step1} 
  Let $H \trianglelefteq_\mathrm{o} G$ and
  $I \subsetneqq \{1, \dots, {\kr}\}$ be such that
  \begin{equation} \label{equ:star-condition}
  \text{$w_J (G; H)=1$ for all $J \subsetneqq I$.} \tag{$\ast$}
  \end{equation}
  Suppose that $\lvert G_w \rvert < 2^{\aleph_0}$.  Let
  $(g_i)_{i\in I }$ be an arbitrary family in~$G$.  Then
  there exists
  $U \trianglelefteq_\mathrm{o} G$, with $U \subseteq H$, such that
  \[
  w_I(g_i; U) =1.
  \]
\end{lemma}

\begin{proof}
  The image of the continuous map
  \[
  H \times \cdots \times H \to G, \quad (h_1, \ldots , h_{\kr})
  \mapsto w_I(g_i h_i; h_i)
  \]
  contains less than $2^{\aleph_0}$ elements.  By
  Proposition~\ref{pro:instead-of-Baire}, there exist
  $b_1, \ldots, b_r \in H$ and $U \trianglelefteq_\mathrm{o} G$, with
  $U \subseteq H$, such that
  \[
  w_I(g_i b_i u_i; b_i u_i) = w_I(g_i b_i ; b_i ) \quad \text{for
    all $u_1, \ldots , u_{\kr} \in U$}. 
  \]
 
  As $I \subsetneqq \{ 1, \dots , {\kr}\}$, we conclude from
  Corollary~\ref{cor:uno} that
  \begin{equation}\label{equ:1-way}
    w_I(g_i b_i U; U)=1.
  \end{equation}
  On the other hand, based on~\eqref{equ:star-condition} and the fact
  that $b_i U\subseteq H$, we deduce from Corollary~\ref{cor:M} that
  \begin{equation}\label{equ:2-way}
    w_I(g_i b_i U; U) = w_I(g_i ; U).
  \end{equation}
  From \eqref{equ:1-way} and~\eqref{equ:2-way} we conclude that
  $w_I(g_i ; U) =1$.
\end{proof}
 
\begin{lemma}\label{basic-step}
  Suppose that $\lvert G_w \rvert < 2^{\aleph_0}$.  Suppose further 
  that $H \trianglelefteq_{\mathrm{o}} G$ satisfies~$w (H)=1$.  Then
  $G_w$ is finite.
\end{lemma}

\begin{proof} 
  Below we construct $V \trianglelefteq_\mathrm{o} G$
   such that
  \begin{equation}\label{equ:**} w_J( G; V ) = 1 \qquad
    \text{for every proper subset $J \subsetneqq \{1, \ldots, \kr\}$.}
  \end{equation}
  Let $S$ be a transversal, i.e., a set of coset representatives, for
  $V$ in~$G$.  From~\eqref{equ:**} and Corollary~\ref{cor:M} we deduce
  that
  \[
  w(g_1 v_1,\ldots,g_{\kr} v_{\kr}) = w(g_1,\ldots,g_{\kr})  \quad
  \text{for all $g_1,\ldots,g_{\kr} \in S$ and all $v_1,\ldots,v_{\kr} \in V$.}
  \]
  Since $G = \bigcup \{ g V \mid g \in S \}$, this shows that $G_w = \{
  w(g_1,\ldots,g_{\kr}) \mid g_1, \ldots, g_{\kr} \in S \}$ is finite.

  \smallskip

  It remains to produce $V \trianglelefteq_\mathrm{o} G$ such that
  \eqref{equ:**} holds.  Indeed, we prove for
  $I \subsetneqq \{1,\ldots,\kr\}$, by induction on $\lvert I \rvert$,
  that there exists $U_I \trianglelefteq_\mathrm{o} G$
    such that $w_I( G; U_I) = 1$.  The group $V$ then
  results from intersecting the finitely many groups $U_I$,
  where $I \subsetneqq \{1, \ldots, r\}$.
 
  Let $I \subsetneqq \{1,\ldots,\kr\}$. If $I = \varnothing$ then
  $U_\varnothing = H$ satisfies
  $w_\varnothing(G ; U_\varnothing) = w(H) =1$.  Now suppose that
  $\lvert I \rvert \ge 1$.  For each $J \subsetneqq I$ induction
  yields $U_J \trianglelefteq_\mathrm{o} G$
  such that $w_J(G; U_J) =1$.  Then
  $U = \bigcap \{ U_J \mid J \subsetneqq I \}
  \trianglelefteq_\mathrm{o} G$
  satisfies
  \begin{equation}\label{equ:*} w_J( G; U )=1\qquad
    \text{for every proper subset $J \subsetneqq I$.}
  \end{equation}

  Let $R$ be a transversal for $U$ in~$G$.  For each family
  $\mathbf{g} = (g_i)_{i \in I}$ in~$R$, Lemma~\ref{step1} yields
  $U_\mathbf{g} \trianglelefteq_\mathrm{o} G$, with
  $U_\mathbf{g} \subseteq U$, such that $w_I(g_i ; U_\mathbf{g} ) =1$.
  Intersecting the finitely many groups $U_\mathbf{g}$, parametrised
  by $\mathbf{g}$, we obtain
  $U_I \trianglelefteq_\mathrm{o} G$, with $U_I \subseteq U$, such
  that
  \[ 
  w_I(g_i ; U_I ) = 1 \qquad \text{for all families $\mathbf{g} = (g_i)_{i \in I}$
    in $R$.}
  \]
  From~\eqref{equ:*} and Corollary~\ref{cor:M} we deduce that
  \[
  w_I(g_i U ; U_I ) = w_I(g_i ; U_I) =1 \qquad \text{for all families
    $\mathbf{g} =  (g_i)_{i \in I}$ in $R$.}
  \]
  Since $G = \bigcup \{ g U \mid g \in R \}$, this shows that 
 $w_I (G ; U_I ) = \langle \bigcup_\mathbf{g} w_I(g_i U; U_I )\rangle
  =1$.
\end{proof}

With these preparations we are ready to prove
Theorem~\ref{thm:mcw-more}.

\begin{proof}[Proof of Theorem~{\rm\ref{thm:mcw-more}}]
  Recall that $w = w(x_1,\ldots,x_r)$ is a multilinear commutator word
  and that $G$ is a profinite group such that
  $\lvert G_w \rvert < 2^{\aleph_0}$.  Clearly, we may assume that
  $r \ge 1$.  By Proposition~\ref{pro:instead-of-Baire} and
  Corollary~\ref{cor:uno}, there exists
  $H \trianglelefteq_\mathrm{o} G$ such that $w(H)=1$.  Thus
  Lemma~\ref{basic-step} shows that $\lvert G_w \rvert < \infty$ and
  the claim follows from~\cite[Theorem~1]{Wi74} (or
  \cite[Theorem~1.1]{DeMoSh16}).
\end{proof}

%%
%%%%% 

\section{The case where $w(G)$ is generated by finitely many
  $w$-values}  \label{sec:fin-many-w-values}

In theory, the task of establishing the strong conciseness of a group
word $w$ for a class $\mathcal{C}$ of profinite groups can be divided
into two steps: Given $w$ and a profinite group $G$ in $\mathcal{C}$
such that $\lvert G_w \rvert < 2^{\aleph_0}$, it suffices to show that
\begin{itemize}
\item[$\circ$] $w(G)$ is generated by finitely many $w$-values and
\item[$\circ$] using this extra information, the group $w(G)$ is
  finite.
\end{itemize}

If $w(G)$ is a pro-$p$ group, for some prime~$p$, the situation
simplifies further: the verbal subgroup $w(G)$ is generated by
finitely many $w$-values if and only if it is finitely generated.
 Indeed, it suffices to look at the Frattini quotient of
  $w(G)$, an elementary abelian pro-$p$ group. In addition, we have
the following useful lemma.

\begin{lemma} \label{lem:reduction-to-pro-p-central-exp-p}
 Let $w$ be a group word and let $G$ be a profinite
    group such that $\lvert G_w \rvert < 2^{\aleph_0}$.  Suppose that
    $N = w(G)$ is a pro-$p$ group, for some prime~$p$, and that
    $N/[N,G]N^p$ is finite.  Then $w(G)$ is generated by finitely many
    $w$-values.
\end{lemma}

\begin{proof}
  Let $X$ be a finite set of $w$-values such that
  $N = \langle X \rangle [N,G] N^p$.  Since $N$ is a pro-$p$ group,
  the set $\{ x^g \mid x \in X, g \in G\}$ generates $N$ modulo
  $N \cap K$ for every open normal subgroup
  $K \trianglelefteq_\mathrm{o} G$.  Hence
  $N = \langle x^g \mid x \in X, g \in G \rangle$, and from
  Lemma~\ref{lem:fin-conj-cl} we conclude that
  $\{ x^g \mid x \in X, g \in G\} $ is finite.
\end{proof}

We recall that a group word $w$ has finite width in an abstract group
$G$ if there exists $m \in \mathbb{N}$ such that every element
$g \in w(G)$ can be written as a product $g = g_1 \cdots g_m$, where
each $g_i$ is a $w$-value or the inverse of a $w$-value in~$G$.
  This notion extends naturally to profinite groups.  If
  $w$ has finite width in a profinite group~$G$, then $w(G)$ coincides
  with the abstract subgroup generated by $G_w$; see
  \cite[Proposition~4.1.2]{Se09}.

\begin{cor} \label{cor:reduction-to-pro-p-central-exp-p} 
    Let $w$ be a group word and let $G$ be a profinite group such that
    $\lvert G_w \rvert < 2^{\aleph_0}$.  Suppose that $w(G)$ is a
    pro-$p$ group, for some prime~$p$, and that $w$ has finite width
    in every finitely generated subgroup $H$ of~$G$.  Then $N = w(G)$
    is finite if and only if $N/[N,G]N^p$ is finite.
\end{cor}

\begin{proof} Suppose that $N/[N,G]N^p$ is finite.  By
  Lemma~\ref{lem:reduction-to-pro-p-central-exp-p}, $N$ is generated
  as a subgroup by finitely many $w$-values.  Thus we may further
  suppose that $G$ is finitely generated.  By our assumptions, $w$ has
  finite width in~$G$.  Hence $\lvert w(G) \rvert < 2^{\aleph_0}$ and
  $w(G)$ is finite.
\end{proof}

We now extend our considerations to general profinite groups~$G$, but
impose a priori the condition that $w(G)$ is generated by finitely
many $w$-values.

\begin{lemma}\label{lem:fin-ab}
  Let $w$ be a group word and let $G$ be a profinite group such that
  $\lvert G_w \rvert < 2^{\aleph_0}$.  Suppose that $w(G)$ is
  generated by finitely many $w$-values.  Then the commutator subgroup
  of $w(G)$ is finite.
\end{lemma} 

\begin{proof}
  Suppose that $w(G) = \langle g_1, \ldots, g_{\kr} \rangle$ for
  $g_1, \ldots, g_{\kr} \in G_w$.  Lemma~\ref{lem:fin-conj-cl} implies
  that $\mathrm{C}_G(g_1)$, \ldots, $\mathrm{C}_G(g_{\kr})$ are open
  in~$G$.  Therefore
 $\mathrm{C}_G(w(G)) = \bigcap_{i=1}^n \mathrm{C}_G(g_i)
    \le_\mathrm{o} G$
  and $G/\mathrm{C}_G(w(G))$ is finite.  By Schur's Theorem (see
  \cite[p.~102]{Ro72}), the commutator subgroup of $w(G)$ is finite.
\end{proof}

  For $n,k \in \mathbb{N}_0$, the $n$th power of the
  derived word $\delta_k$ is written as~$\delta_k^{\, n}$.  We say
  that a quantity is $(a,b,c,\ldots)$-bounded if it can be bounded
  from above by a number depending only on the specified parameters
  $a,b,c,\ldots$. 

\begin{lemma}\cite[Lemma~3.2]{Sh02}\label{lem:3.2shu2}
  Let $k,n,t \in \mathbb{N}$.  Let $G$ be a group satisfying
  $\delta_k^{\, n}(G) = 1$.  Let $H$ be a nilpotent subgroup of $G$
  generated by a set of $\delta_k$-values and suppose, in addition,
  that $H$ is $t$-generated. Then the order of $H$ is
  $(k,n,t)$-bounded.
\end{lemma} 

\begin{lemma}\cite[Lemma~2.1]{FeSh18}\label{lem:2.1FeSh}
  Let $d, k \in \mathbb{N}$.  There exists a number $t=t(d,k)$,
  depending on $d$ and $k$ only, such that, if $G$ is a finite
  $d$-generated group, then every $\delta_{k-1}$-value in elements of
  $G'$ is a product of at most $t$ elements that are $\delta_k$-values
  in elements of~$G$.
\end{lemma} 

\begin{lemma}\cite[Lemma~2.2]{FeSh18}\label{lem:2.2FeSh}
  Let $G$ be a soluble group of derived length~$l$, and suppose that
  $X$ is a symmetric, normal and commutator-closed set of generators
  for~$G$.  Let $g$ be an arbitrary element of $G$, written as
  $g = x_1 \cdots x_t$, where $x_i \in X$ for all
  $i \in \{1, \ldots t\}$.  Then, for every $n \in \mathbb{N}$, we
  have
  \[
  g^{n^l} = y_1^{\, n} \cdots y_s^{\, n},
  \]
  where $y_1, \ldots, y_s \in X$ and $s$ is $(n,t,l)$-bounded. 
\end{lemma}

  \begin{prop} \label{pro:delta-k-q-result} Let $w = \delta_k^{\, n}$,
    where $k, n \in \mathbb{N}_0$.  Let $G$ be a profinite group such
    that $\lvert G_w \rvert < 2^{\aleph_0}$.  Suppose that the $k$th
    derived subgroup $G^{(k)}$ is pronilpotent and that $w(G)$ is
    finitely generated.  Then $w(G)$ is finite.
\end{prop} 

\begin{proof} We argue by induction on~$n$.  For $n\leq1$, the result
  is immediate from Theorem~\ref{thm:mcw-more}.  Now suppose that
  $n\geq2$.  Let $\pi$ be the set of prime divisors of~$n$.  If
  $G^{(k)}$ is a pro-$p'$ group for some $p \in \pi$, then for
  $v = \delta_k^{\, n/p}$ the map $x \mapsto x^p$ provides a bijection
  from $G_v$ onto $G_w$ and, by induction $w(G) = v(G)$ is finite.
  Hence, we may suppose that $G^{(k)}$ has non-trivial Sylow pro-$p$
  subgroup for each $p \in \pi$.  Moreover, if $P$ is the Sylow
  pro-$p$ subgroup of $G^{(k)}$ for some $p\in\pi$, then the image of
  $w(G)$ in $G/P$ is finite.  Suppose that $\lvert \pi \rvert \ge 2$,
  and let $P_1$ and $P_2$ be the Sylow subgroups of $G^{(k)}$ for
  distinct primes $p_1, p_2 \in \pi$.  Then images of $w(G)$ in
  $G/P_1$ and $G/P_2$ are finite, and from $P_1\cap P_2=1$ we deduce
  that $w(G)$ is finite.

  Thus, it is sufficient to deal with the case where $n$ is a
  $p$-power for some prime~$p$.  Passing to the quotient
  $G/O_{p'}(G^{(k)})$, we may suppose that $G^{(k)}$ is a pro\nobreakdash-$p$
  group.  Since $w(G)$ is a finitely generated pro-$p$ group, it is
  actually generated by finitely many $w$-values.  By
  Lemma~\ref{lem:fin-ab}, the commutator subgroup of $w(G)$ is finite.
  Passing to the quotient $G/w(G)'$, we may suppose that $w(G)$ is
  abelian.  If $k=0$, we deduce from Lemma~\ref{lem:not-comm-periodic}
  that $G$ is periodic, hence $w(G)$ is finite.

  Suppose that $k\ge1$.  Since $w(G)$ is generated by finitely many
  $w$-values, we may choose finitely many elements
  $g_1,\dots,g_d\in G$ such that
  $w(\langle g_1,\dots,g_d\rangle)=w(G)$. It is sufficient to work
  with $\langle g_1,\dots,g_d\rangle$ in place of $G$ and so without
  loss of generality we suppose that $G$ is finitely generated, by $d$
  elements, say.  By Lemma~\ref{lem:2.1FeSh} there exists a number
  $t=t(d,k)$, depending on $d$ and $k$ only, such that every
  $\delta_{k-1}$-value in elements of $G'$ is a product of at most $t$
  elements which are $\delta_{k}$-values in elements of $G$.

  Consider a subgroup $H = \langle x_1, \ldots , x_t\rangle$, where
  $x_1, \ldots, x_t$ are $\delta_k$-values in~$G$.  By
  Lemma~\ref{lem:3.2shu2}, applied to finite quotients of $G/w(G)$,
  every finite quotient of $H/(H \cap w(G))$ and hence the entire
  group $H/(H \cap w(G))$ is finite of $(k,n,t)$-bounded order.  In
  particular, $H$ is soluble of derived length at most $l$, where
  $l=l(k,n,t)$ depends on $k, n, t$ only.

  Set $v= (\delta_{k-1})^{n^l}$. 
  By Lemma~\ref{lem:2.2FeSh}, every
  $v$-value in elements of $G'$ is a product of an $(n,t,l)$-bounded
  number of $w$-values and inverses of $w$-values.  This gives
  $\lvert (G')_v \rvert < 2^{\aleph_0}$ and, by induction on $k$, the
  verbal subgroup $v(G')$ is finite.
 
  Passing to the quotient $G/v(G')$, we may suppose that
  $\delta_{k-1}$-values in elements of $G'$ are of finite order.  Then
  also $\delta_k$-values in elements of $G$ are of finite order.  As
  $w(G)$ is abelian and finitely generated, we conclude that $w(G)$ is
  finite.
\end{proof}

Recall that a group word $w$ is a \emph{law} in a group $G$ if
$w(G)=1$.  We say that $w$ \emph{implies virtual nilpotency} if every
finitely generated metabelian group for which $w$ is a law has a
nilpotent subgroup of finite index.  Burns and Medvedev~\cite{BuMe03}
showed that if $w$ implies virtual nilpotency, then for a much larger
class of groups~$G$, including all finitely generated residually
finite groups, $w(G) = 1$ implies that $G$ is
nilpotent-by-finite. Moreover, the word $w$ implies virtual nilpotency
if and only if, for all primes $p$, the word $w$ is not a law in the
wreath product $C_p \wr C_\infty$ of the cyclic group of order $p$ by
the infinite cyclic group; see~\cite{BuMe03}.  In particular, every
word of the form $uv^{-1}$, where $u$ and $v$ are positive words
(i.e.\ semigroup words in finitely many free generators), implies
virtual nilpotency.  Furthermore, by a result of
Gruenberg~\cite{Gr53}, all Engel words imply virtual nilpotency.
Other examples of words implying virtual nilpotency include
generalisations of Engel words, such as words of the form
$w = w(x,y) = [x^{e_1},y^{e_2},\ldots,y^{e_{\kr}}]$, where
$\kr \in \mathbb{N}$ and
$e_1,\ldots,e_{\kr} \in \mathbb{Z} \smallsetminus \{0\}$. To see that
such a word implies virtual nilpotency, we employ the criterion of
Burns and Medvedev.  The case $r=1$ is easy; now suppose that
$r \ge 2$.  Let $p$ be a prime and consider the wreath product
\[
C_p \wr C_\infty = \langle a, t \mid a^p = 1, [a^{t^i},a^{t^j}] = 1
\text{ for } i,j \in \mathbb{Z} \rangle
\overset{\cong}{\longrightarrow} \mathbb{F}_p[T,T^{-1}] \rtimes
\langle T \rangle,
\]
where $a \mapsto 1 \in \mathbb{F}_p[T,T^{-1}]$ in the base group and
$t \mapsto T \in \langle T \rangle \le \mathbb{F}_p[T,T^{-1}]^*$ in
the top group.  We may suppose that~$e_1>0$.  Then the indicated
isomorphism maps $w(ta,t)$ to
\[
(T^{e_1 - 1} + \ldots + T + 1)(T^{e_2} - 1) \cdots (T^{e_r} - 1) \ne
0,
\]
in the base group.  Thus $w(ta,t) \ne 1$ and $w$ is not a law in
$C_p \wr C_\infty$.

\begin{prop} \label{pro:virt-nilp}
  Let $w$ be a word implying virtual nilpotency and let $G$ be a
  profinite group such that $\lvert G_w \rvert < 2^{\aleph_0}$. If the
  verbal subgroup $w(G)$ is generated by finitely many $w$-values,
  then $w(G)$ is finite.
\end{prop} 

\begin{proof}
  Without loss of generality we may assume that $G$ is finitely generated. Using Lemma~\ref{lem:fin-ab}, we may further assume that $w(G)$ is abelian.
  Clearly, $w$ is a law in $G/w(G)$. Hence \cite[Theorem~A]{BuMe03}
  shows that $G/w(G)$ is nilpotent-by-finite.  Thus $G$ is
  abelian-by-nilpotent-by-finite.  Every word has finite width in
  every finitely generated abelian-by-nilpotent-by-finite group;
  compare~\cite[Theorem~4.1.5]{Se09}.  
  
  Thus $w(G)$ has less than $2^{\aleph_0}$ elements, hence it is
  finite.
\end{proof}

Following~\cite{GuSh15} we say that a group word $w$ is \emph{weakly
  rational} if for every finite group $G$ and for every positive integer $e$ with $\gcd(e,\lvert G \rvert) = 1$, the set $G_w$
is closed under taking $e$th powers of its elements.
By~\cite[Lemma~1]{GuSh15}, the word $w$ is weakly rational if and only
if for every finite group $G$, every $g \in G_w$ and every
$e \in \mathbb{N}$ with $\gcd(e,\lvert \langle g \rangle \rvert) =1$ we
have $g^e \in G_w$.  According to \cite[Theorem~3]{GuSh15}, the word
$w=[x_1,\ldots,x_{\kr}]^q$ is weakly rational for all
$\kr,q \in \mathbb{N}$.

\begin{prop} \label{pro:weakly-rational} 
  Let $w$ be a weakly rational word and let $G$ be a profinite group
  such that $\lvert G_w \rvert < 2^{\aleph_0}$. If the verbal subgroup
  $w(G)$ is generated by finitely many $w$-values, then $w(G)$ is
  finite.
\end{prop} 

\begin{proof}
  By Lemma \ref{lem:fin-ab} we may suppose that $w(G)$ is abelian, and
  it suffices to show that elements of $G_w$ have finite order.
  
  Let $h \in G_w$, and let $g$ be any generator of the procyclic group
  $H = \langle h \rangle$.  For every
  $N \trianglelefteq_\mathrm{o} G$, there exists $e \in \mathbb{N}$
  with $\gcd(e, \lvert HN/N \rvert|) = 1$ such that $g \equiv_N h^e$
  and, because $w$ is weakly rational, we obtain $g \in G_w N$.  Hence
  $g \in \bigcap_{N \trianglelefteq_\mathrm{o} G} G_w N = G_w$.
  Therefore the procyclic group $H$ has less than $2^{\aleph_0}$
  single generators.

  This implies that $H$ is finite. Indeed, consider the Frattini subgroup
  $\Phi(H)$ of~$H$.  Since $\langle g \rangle = H$ for every $g \in h\Phi(H)$, the group $\Phi(H)$ has less than
  $2^{\aleph_0}$ elements.  Hence $\Phi(H)$ is finite, and without
  loss of generality we assume that $\Phi(H) = 1$.  Then
  $H \cong \prod_{p \in\pi} C_p$ for a set of primes~$\pi$.  Each factor
  $C_p$ has $p-1$ single generators.  Since $H$ has less than
  $2^{\aleph_0}$ single generators, $\pi$ and hence $H$ is
  finite.
\end{proof}

%%%%%
%%%%%

\section{Reduction via parametrised words} \label{sec:para-words}

Throughout this section, we fix a profinite group~$G$, a positive
integer $r \in \mathbb{N}$ and a normal subgroup
$\mathbf{G} \trianglelefteq G \times \ldots \times G$ of the direct
product of $r$ copies of~$G$.  A typical situation would be
$\mathbf{G} = G_1 \times \cdots \times G_r$, where
$G_1, \ldots , G_r \trianglelefteq G$.

Our intention is to consider (products of) `parametrised group words'
in variables $x_1, \ldots, x_r$, with parameters coming from~$G$ where
each $x_i$ is intended to take values in~$G_i$ and where we formally
distinguish repeated occurrences of the same variable.  This
elementary concept requires a flexible but precise set-up.

Let $\Omega = \Omega_{G,r}$ be the free group on free generators
\[
\xi_h, \quad \eta_{1,i}, \eta_{2,i}, \ldots, \eta_{r,i} \qquad
\text{for $h \in G$ and $i \in \mathbb{N}$.}
\]
Informally, we think of each free generator $\xi_h$ as a `parameter
variable' that is to take the value $h$ and each free generator
$\eta_{q,i}$ as a `free variable' that can be specialised to~$x_q$,
irrespective of the additional index~$i$.

We refer to elements $\omega \in \Omega$ as \emph{$r$-valent
  parametrised words} for~$G$ or, since $r$ is fixed throughout,
simply as \emph{parametrised words} for~$G$.  For
$\mathbf{g} = (g_1,\ldots,g_r) \in \mathbf{G}$, we write 
\[
\uo(\mathbf{g}) = \uo_\mathbf{G}(\mathbf{g}) = \uo(g_1, \ldots, g_r)
\in G
\]
for the $\omega$-value that results from replacing each $\xi_h$ by $h$
and each $\eta_{q,i}$ by $g_q$, for all $h \in G$,
$q \in \{1,\ldots,r\}$ and $i \in \mathbb{N}$.  In this way we obtain
a parametrised word map $\uo(\cdot) \colon \mathbf{G} \to G$.

The \emph{degree} $\deg(\omega)$ of the parametrised word $\omega$ is
the number of free generators $\eta_{q,i}$, with
$q \in \{1,\ldots,r\}$ and $i \in \mathbb{N}$, appearing in (the
reduced form of)~$\omega$; here we care whether a generator
$\eta_{q,i}$ appears, but not whether it appears repeatedly.  The
degree $\deg(\omega)$ is a non-negative integer and plays a role in
defining appropriate induction parameters.  We remark that, if
$\omega$ has degree~$0$, then the map $\uo( \cdot) $ is constant,
i.e.\ there exists $h \in G$ such that for all
$\mathbf{g} \in \mathbf{G}$ we have $\uo(\mathbf{g}) = h$.

\begin{example} \label{exa:basic-illustration-pw} Our main interest
  will be in iterated commutator words, such as
  $w(x_1,x_2,x_3) = [[x_1,x_2,x_2],[x_2,x_3]]$, and the $w$-values in
  a profinite group~$G$.  We set $r = 3$,
  $\mathbf{G} = G \times G \times G$ and
  $\omega =
  [[\eta_{1,1},\eta_{2,1},\eta_{2,2}],[\eta_{2,3},\eta_{3,1}]]$
  to model~$w$ in the sense that
  \[
  w(g_1,g_2,g_3) = [[g_1,g_2,g_2],[g_2,g_3]] = \uo(g_1,g_2,g_3) \quad
  \text{for all $g_1,g_2,g_3 \in G$.}
  \]
  The $3$-valent parametrised word $\omega$ has degree~$5$; moreover,
  $\omega$ is a multilinear commutator word of weight~$5$ (meaning
  that it involves $5$ variables).  In this example, we are not yet
  using the possibility to involve parameters.
\end{example}

We fix a set $\mathfrak{E} = \mathfrak{E}_{G,r} \subseteq \Omega$ of
$r$-valent parametrised words for~$G$, which we think of as
`elementary' words, and we consider finite products of such.  To write
down these products we use finite index sets
$T, S, \ldots \subseteq \mathbb{N}$ that are implicitly ordered so
that the products are unambiguous in a typically non-commutative
setting.

Formally, an $r$-valent \emph{$\mathfrak{E}$-product} for~$G$ is a
finite sequence $(\epsilon_t)_{t \in T}$, where
$\epsilon_t \in \mathfrak{E}$ for each $t \in T$; more suggestively,
we denote it by
\[
\dotprod\nolimits_{t \in T} \epsilon_t,
\]
where the dot indicates that we consider a formal product and not the
parametrised word that results from actually carrying out the
multiplication in~$\Omega$.

By a \emph{length function} on $\mathfrak{E}$ we mean any map
$\ell \colon \mathfrak{E} \to \mathsf{W}$ from $\mathfrak{E}$ into a
well-ordered set~$\mathsf{W} = (\mathsf{W},\le)$ such that elements
$\epsilon \in \mathfrak{E}$ whose length $\ell(\epsilon)$ is minimal
with respect to $\le$ also have minimal degree $\deg(\epsilon) = 0$.
As usual, we agree that the maximum of the empty subset of
$\mathsf{W}$ is the least element of~$\mathsf{W}$.  A length function
$\ell$ induces a total pre-order $\preceq_\ell$ on the set of all
$r$-valent $\mathfrak{E}$-products, as follows:
\[
\dotprod\nolimits_{s \in S} \widetilde{\epsilon}_s \;\preceq_\ell\;
\dotprod\nolimits_{t \in T} \epsilon_t \qquad \text{if} \quad \max
\{ \ell(\widetilde{\epsilon_s}) \mid s \in S \} \le \max \{ \ell(\epsilon_t) \mid
t \in T \};
\]
we write
$\dotprod_{s \in S} \widetilde{\epsilon}_s \prec_\ell \dotprod_{t \in
  T} \epsilon_t$
if
$\max \{ \ell(\widetilde{\epsilon}_s) \mid s \in S \} < \max \{ \ell(\epsilon_t)
\mid t \in T \}$.
Clearly, there are no infinite descending chains of
$\mathfrak{E}$-products, with respect to~$\prec_\ell$.  This fact
allows us to give the following recursive definition.

\begin{defn}[Friendly products] \label{def:good-pw} Let
  $\ell \colon \mathfrak{E} \to \mathsf{W}$ be a length function.  We
  define recursively the set
  $\mathfrak{F} = \mathfrak{F}_{\mathbf{G},r,\ell}$ of
  \emph{$\ell$-friendly} $r$-valent $\mathfrak{E}$-products for
  $\mathbf{G}$ as follows.  An $r$-valent $\mathfrak{E}$-product
  $\dotprod_{t \in T} \epsilon_t$ for~$G$ belongs to $\mathfrak{F}$ if
  for every $\mathbf{b} \in \mathbf{G}$ there exists an $r$-valent
  $\mathfrak{E}$-product
  $\dotprod_{s \in S(\mathbf{b})} \widetilde{\epsilon}_{\mathbf{b},s}$
  such that
  \begin{enumerate}
  \item[(F1)\ ]
      $\dotprod_{s \in S(\mathbf{b})}
      \widetilde{\epsilon}_{\mathbf{b},s}$
      belongs to $\mathfrak{F}$ and
      $\dotprod_{s \in S(\mathbf{b})} \widetilde{\epsilon}_{\mathbf{b},s} \prec_\ell
      \dotprod_{t \in T} \epsilon_t$ and \\[-8pt]
    \item[(F2)\ ] the parametrised words
      $\omega = \prod_{t \in T} \epsilon_t$ and
      $\nu_\mathbf{b} = \prod_{s \in S(\mathbf{b})}
      \widetilde{\epsilon}_{\mathbf{b},s}$ satisfy
      \begin{equation} \label{equ:multilin-decomp} \uo(
        \mathbf{b} \mathbf{g}) = \uo(\mathbf{b}) \cdot
        \uo(\mathbf{g}) \cdot
        \underline{\nu}_\mathbf{b}(\mathbf{g}) \qquad \text{for all
          $\mathbf{g} \in \mathbf{G}$.}
      \end{equation}
    \end{enumerate}
\end{defn}

\begin{rem} \label{rem:good-para-words} (1) In the definition, the
  product $\dotprod_{s \in S(\mathbf{b})} \widetilde{\epsilon}_{\mathbf{b},s}$ is
  allowed to be empty, in which case \eqref{equ:multilin-decomp}
  simplifies to
  \begin{equation} \label{equ:simple-hom-cond}
    \uo(\mathbf{b} \mathbf{g}) =
    \uo(\mathbf{b}) \cdot \uo(\mathbf{g}).
  \end{equation}
  Such a strong relation holds, for instance, if the parametrised word
  $\omega = \prod_{t \in T} \epsilon_t$ has degree~$1$ and defines a
  homomorphism~$\mathbf{G} \to G$ that factors through the $q$th
  coordinate, if the single free variable occurring in $\omega$ is
  $\eta_{q,i}$ for some $i \in \mathbb{N}$.  In this special
  situation, \eqref{equ:simple-hom-cond} holds uniformly for all
  $\mathbf{b} \in \mathbf{G}$.

  (2) If the $\ell$-friendly $\mathfrak{E}$-product
  $\dotprod_{t \in T} \epsilon_t$ is minimal with respect
  to~$\preceq_\ell$, then $\dotprod_{s \in S(\mathbf{b})} \widetilde{\epsilon}_{\mathbf{b},s}$  
  is necessarily empty for every choice of
  $\mathbf{b} \in \mathbf{G}$.  Furthermore, each $\epsilon_t$ has
  degree $\deg(\varepsilon_t) = 0$, so there is $h \in G$ such that
  for all $\mathbf{g} \in \mathbf{G}$ we have
  $\uo(\mathbf{g}) = \prod_{t \in T}
  \ue_t(\mathbf{g}) = h$.
  Thus~\eqref{equ:multilin-decomp} yields
  $h = h \cdot h \cdot 1 = h^2$, and hence $h=1$.

  In this sense there is only one parametrised word map coming from an
  $\ell$-friendly $\mathfrak{E}$-product for $\mathbf{G}$ that is
  minimal with respect to~$\preceq_\ell$, namely the constant map with
  value~$1$.  In particular, for every $\ell$-friendly
  $\mathfrak{E}$-product $\dotprod_{t \in T} \epsilon_t$ that is
  second smallest with respect to~$\preceq_\ell$, the parametrised
  word $\omega = \prod_{t \in T} \epsilon_t$
  satisfies~\eqref{equ:simple-hom-cond}.
\end{rem}

\begin{rem} \label{rem:E-ell-for-nilpotent} In this paper we use the
  terminology introduced above in the context of nilpotent groups.  We
  indicate how the general set-up specialises.

  Let $G$ be a nilpotent profinite group of class at most~$c$, i.e.\
  $\gamma_{c+1}(G) = 1$.   Denote by $\mathfrak{E}$ the set of all
  left-normed repeated commutators in the free generators $\xi_h$ and
  $\eta_{q,i}$ of $\Omega$, 
   subject to the restriction
    that each $\eta_{q,i}$ appears at most once.
  In other words, $\mathfrak{E}$ consists of all $\gamma_m$-values,
  for $m \ge 2$, that result from replacing the $m$ variables in
  $\gamma_m = [x_1,x_2,\ldots,x_m]$ by arbitrary free generators
  $\xi_h$ and $\eta_{q,i}$ of $\Omega$,  subject to the
    restriction that each $\eta_{q,i}$ appears at most once.

  For instance, given some element $a \in G$,
   \[ 
  \epsilon_1 = [\xi_a,\eta_{1,1}, \xi_a, \eta_{2,1}, \eta_{2,2},
  \eta_{2,3 }] \in \mathfrak{E},
   \]
   whereas
   $\epsilon_2 = [\xi_a,\eta_{1,1}, \xi_a, \eta_{2,1}, \eta_{2,2},
   \eta_{2,1}]$
   does not lie in ~$\mathfrak{E}$, even though
   $\ue_1(\cdot) =\ue_2(\cdot)$.  We set
   $\mathsf{W} = \mathbb{N}_0 \times \mathbb{N}_0$, equipped with the
   lexicographic order~$\le$; so, for instance, $(2,10) \le (3,0)$ and
   $(5,4) \le (5,7)$.
  
   Every $\epsilon \in \mathfrak{E}$, by definition, belongs to
   $\gamma_2(\Omega)$.  Let $k(\epsilon)$ denote the maximal
   $j \in \{1, \ldots , c+1\}$ such that
   $\epsilon \in \gamma_j(\Omega)$, and define a length function on
   $\mathfrak{E}$ by associating to $\epsilon$ the length
   \[
   \ell(\epsilon) = \big( c+1-k(\epsilon),\deg(\epsilon) \big) \in
   \mathsf{W}.
   \]
   For instance, if $c = 8$ then
   $\ell(\epsilon_1) = (8+1-6,4) = (3,4)$.
\end{rem}

\begin{lemma} \label{lem:write-as-elem-prod} Suppose that the
  profinite group $G$ is nilpotent of class at most~$c$, and let
  $\mathfrak{E}$ be defined as in
  Remark~\textup{\ref{rem:E-ell-for-nilpotent}}.  Let
  $\omega \in \gamma_k(\Omega)$, where $k \ge 2$, be such that
  $\uo(\mathbf{1}) = 1$.  Then there exists an $\mathfrak{E}$-product
  $\dotprod_{t \in T} \epsilon_t$, with
  $\epsilon_t \in \gamma_k(\Omega)$ and $\deg(\epsilon_t) \ge 1$ for
  all $t \in T$, such that
  \[
  \uo( \mathbf{g}) = \prod\nolimits_{t \in T} \ue_t
  ( \mathbf{g}) \qquad \text{for all
    $\mathbf{g} \in \mathbf{G}$.}
  \]
\end{lemma}

\begin{proof} 
  It suffices to prove, by induction on $c+1-k$, that
  \begin{equation}\label{equ:omega-decomposition}
    \omega \equiv \prod_{\substack{t \in T \text{ s.t.} \\
        \deg(\epsilon_t)=0}} \epsilon_t \cdot \prod_{\substack{t \in T
        \text{ s.t.} \\ \deg(\epsilon_t) \neq 0}} \epsilon_t \mod
    \gamma_{c+1}(\Omega).
  \end{equation}
  for a suitable index set $T$ and suitable left-normed repeated
  commutators
  $\epsilon_t = [\varkappa_{t,1}, \ldots, \varkappa_{t,n(t)}]$, where
  $n(t) \ge k$ and the terms $\varkappa_{t,j}$ stand for suitable free
  generators $\xi_h$ and $\eta_{q,i}$ of~$\Omega$.  Indeed, using the
  infinite supply of generators $\eta_{q,i}$, we can rename the free
  generators entering into the commutators $\epsilon_t$ to ensure that
  $\epsilon_t \in \mathfrak{E}$, without changing the resulting word
  maps~$\ue_t(\cdot)$.  Furthermore, the identity
  \[
  1= \uo( \mathbf{1}) = \prod_{\substack{t \in T \text{ s.t.} \\
      \deg(\epsilon_t)=0}} \ue_t ( \mathbf{1}) \cdot
  \prod_{\substack{t \in T \text{ s.t.} \\ \deg(\epsilon_t)>0}} \ue_t
  ( \mathbf{1}) = \prod_{\substack{t \in T \text{ s.t.} \\
      \deg(\epsilon_t)=0}} \ue_t ( \mathbf{1}) 
  \]
  and the fact that $\ue_t(\cdot)$ is constant whenever
  $ \deg(\epsilon_t)=0$, shows that
  \[
  \uo( \mathbf{g}) = \prod_{\substack{t \in T \text{ s.t.} \\
      \deg(\epsilon_t)>0}} \ue_t ( \mathbf{g}) \qquad \text{for all
    $\mathbf{g} \in \mathbf{G}$.}
  \]

  For $k=c+1$ the congruence~\eqref{equ:omega-decomposition} holds
  upon setting $T=\varnothing$.  Now suppose that $k < c+1$.  As
  $\omega \in \gamma_k(\Omega)$ is a product of $\gamma_k$-values in
  $\Omega$, basic commutator manipulations
  (compare~\cite[Proposition~1.2.1]{Se09}) yield that $\omega $ can be
  written as a product
  \[ 
  \omega = \prod\nolimits_{t \in T(1)} \epsilon_t \cdot \nu
  \]
  of repeated commutators $\epsilon_t$ of the form
  $[\varkappa_{t,1}, \ldots, \varkappa_{t,k}]$ or
  $[\varkappa_{t,1}, \ldots , \varkappa_{t,k}]^{-1}$, where the terms
  $\varkappa_{t,j}$ stand for suitable free generators of~$\Omega$,
  and an element $\nu \in \gamma_{k+1}(\Omega)$.  Modulo
  $\gamma_{k+1}(\Omega)$, the basic relation
  \[ 
  [\varkappa_1,\varkappa_2,\varkappa_3,\ldots,\varkappa_k]^{-1} \equiv
  [[\varkappa_1,\varkappa_2]^{-1},\varkappa_3,\ldots,\varkappa_k]
  \equiv [\varkappa_2,\varkappa_1,\varkappa_3,\ldots,\varkappa_k]
  \]
  holds; thus we can even avoid using terms with exponent~$-1$.

  By induction, $\nu$ can be written as a product 
  \[
  \nu \equiv \prod\nolimits_{t \in T(2)} \epsilon_t \mod
  \gamma_{c+1}(\Omega). 
  \]
  of suitable repeated commutators.  Denote by $T$ the ordinal sum of
  $T(1)$ and $T(2)$, i.e.\ the disjoint union equipped with the total
  order in which every $t_1 \in T(1)$ precedes every $t_2 \in T(2)$
  and where $T(1)$ and $T(2)$ are ordered as before.  This yields
   \[
   \omega \equiv \prod\nolimits_{t \in T} \epsilon_t \mod
   \gamma_{c+1}(\Omega).
   \]
   At the expense of creating extra factors of degree at least~$1$, we
   can rearrange the factors in the product so that, after enlarging
   the index set and renaming the relevant factors, we arrive
   at~\eqref{equ:omega-decomposition}.
\end{proof}

\begin{lemma} \label{lem:formal} Suppose that the profinite group $G$
  is nilpotent of class at most~$c$, and let $\mathfrak{E}$ be defined
  as in Remark~\textup{\ref{rem:E-ell-for-nilpotent}}.  Let $k \ge 2$
  and let $\dotprod_{t \in T} \epsilon_t$ be an
  $\mathfrak{E}$-product, where $\epsilon_t \in \gamma_k(\Omega)$ and
  $\deg (\epsilon_t) \ge 1$ for all $t \in T$.  Let
  $\mathbf{b} = (b_1,\ldots,b_r) \in \mathbf{G}$.  Then there exists a
  parametrised word $\nu = \nu_\mathbf{b} \in \gamma_{k+1}(\Omega)$
  such that
  \[
  \prod_{t \in T} \ue_t( \mathbf{b} \mathbf{g})= \prod_{t \in T}
  \ue_t( \mathbf{b} )\cdot \prod_{t \in T} \ue_t( \mathbf{g}) \cdot
  \prod_{t \in T} \tilde\ue_t (\mathbf{b} /\!\!/ \mathbf{g}) \cdot
  \underline\nu (\mathbf{g}) \qquad \text{for all
    $\mathbf{g} \in \mathbf{G}$,}
  \]
  where $ \tilde \epsilon_t (\mathbf{b} /\!\!/ \cdot)$ denotes the
  $\mathfrak{E}$-product involving (in some implicit order) the
  $2^{\deg(\epsilon_t)} - 2$ factors that result from $\epsilon_t$ by
  replacing a selection of at least one, but not all distinct free
  variables $\eta_{q,i}$ occurring in $\epsilon_t$ by~$\xi_{b_q}$.
  Moreover,
  \[
  \dotprod\nolimits_{t \in T} \tilde\epsilon_t (\mathbf{b} /\!\!/
  \cdot) \prec_\ell \dotprod\nolimits_{t \in T} \epsilon_t \qquad
  \text{and} \qquad \underline\nu (\mathbf{1}) = 1.
  \] 
\end{lemma}

\begin{proof}
  As $\deg (\epsilon_t) \ge 1$, basic commutator manipulations
  (compare~\cite[Proposition~1.2.1]{Se09}) yield that, for
  each~$t \in T$, there exists $\nu_t \in \gamma_{k+1}(\Omega)$ such
  that
  \[ 
  \ue_t ( \mathbf{b} \mathbf{g}) = \ue_t ( \mathbf{b} ) \, \ue_t
  (\mathbf{g}) \, \tilde\ue_t (\mathbf{b} /\!\!/ \mathbf{g}) \,
  \underline\nu_t (\mathbf{g}) \qquad \text{for all
    $\mathbf{g} \in \mathbf{G}$.}
  \]
  Moreover, by construction each factor of the $\mathfrak{E}$-product
  $\tilde\epsilon_t (\mathbf{b}/\!\!/ \cdot)$ has degree at least~$1$,
  hence $\tilde\ue_t (\mathbf{b}/\!\!/ \mathbf{1})=1$.

All the words $\epsilon_t$,
  $\tilde\epsilon_t (\mathbf{b}/\!\!/ \cdot)$ and $\nu_t$, for
  $t \in T$, commute with one another modulo $\gamma_{k+1}(\Omega)$.
  Hence there exists $\nu \in \gamma_{k+1}(\Omega)$ such that, for all
  $\mathbf{g} \in \mathbf{G}$,
  \begin{equation} \label{equ:exp-e-ps-nu}
    \begin{split}
      \prod_{t \in T} \ue_t( \mathbf{b} \mathbf{g}) %
      & = \prod_{t \in T} \big( \ue_t (\mathbf{b} ) \,
      \ue_t(\mathbf{g}) \, \tilde\ue_t (\mathbf{b}/\!\!/\mathbf{g}) \,
      \underline\nu_t (\mathbf{g}) \big) \\
      & = \prod_{t \in T} \ue_t( \mathbf{b} )\cdot \prod_{t \in T}
      \ue_t( \mathbf{g}) \cdot \prod_{t \in T} \tilde\ue_t
      (\mathbf{b}/\!\!/ \mathbf{g}) \cdot \underline\nu (\mathbf{g}).
    \end{split}
  \end{equation}
  Since every factor of $\tilde\epsilon_t (\mathbf{b}/\!\!/ \cdot)$
  has degree strictly smaller than $\epsilon_t$, we deduce that
  \[
  \dotprod\nolimits_{t \in T} \tilde\epsilon_t (\mathbf{b}/\!\!/
  \cdot) \prec_\ell \dotprod\nolimits_{t \in T} \epsilon_t.
  \]
  Finally, substituting $\mathbf{1}$ for $\mathbf{g}$
  in~\eqref{equ:exp-e-ps-nu}, we see that
  $\underline\nu (\mathbf{1})=1$.
\end{proof} 

\begin{lemma} \label{rem:gamma_c} Suppose that the profinite group $G$
  is nilpotent of class at most~$c$, and let $\mathfrak{E}$ and $\ell$
  be defined as in Remark~\textup{\ref{rem:E-ell-for-nilpotent}}.  Let
  $\omega \in \gamma_k(\Omega)$, where $k \ge 2$, be such that
  $\uo(\mathbf{1}) = 1$.  Then there exists an $\ell$-friendly
  $\mathfrak{E}$-product $\dotprod_{t \in T} \epsilon_t$ such that
  \[
  \uo( \mathbf{g}) = \prod\nolimits_{t \in T} \ue_t
  ( \mathbf{g}) \qquad \text{for all
    $\mathbf{g} \in \mathbf{G}$.}\]
\end{lemma}

\begin{proof}
  We argue by induction on~$c+1-k$.  If $c+1-k \le 0$, then
  $\uo(\cdot)$ is the constant map with value~$1$, and the assertion
  holds trivially; compare Remark~\ref{rem:good-para-words}.

  Now suppose that $c +1-k \ge 1$.  We may suppose, in addition, that
  $\omega \not \in \gamma_{k+1}(\Omega)$.
  By Lemma~\ref{lem:write-as-elem-prod}, there exists an
  $\mathfrak{E}$-product $\dotprod_{t \in T} \epsilon_t$, with
  $\epsilon_t \in \gamma_k(\Omega)$ and $\deg(\epsilon_t) \ge 1$ for
  all $t \in T$, such that
  \[
  \uo( \mathbf{g}) = \prod\nolimits_{t \in T} \ue_t ( \mathbf{g})
  \qquad \text{for all $\mathbf{g} \in \mathbf{G}$.}
  \]

  We claim that $\dotprod_{t \in T} \epsilon_t$ is $\ell$-friendly;
  our task is to check the conditions laid out in
  Definition~\ref{def:good-pw}.  We argue by induction with respect to
  the pre-order~$\preceq_\ell$.  Let $\mathbf{b} \in \mathbf{G}$.
  By Lemma~\ref{lem:formal}, there exists a parametrised word
  $\nu \in \gamma_{k+1}(\Omega)$ such that
  \begin{equation}\label{equ:e1}
    \prod_{t \in T} \ue_t( \mathbf{b}  \mathbf{g})=
    \prod_{t \in T} \ue_t( \mathbf{b} )\cdot \prod_{t \in T}
    \ue_t(\mathbf{g}) \cdot  \prod_{t \in T} \tilde\ue_t
    (\mathbf{b}/\!\!/
    \mathbf{g})  
    \cdot \underline\nu (\mathbf{g}) 
    \qquad \text{for all
      $\mathbf{g} \in \mathbf{G}$,}
  \end{equation}
  where the $\mathfrak{E}$-product
  $\dotprod_{t \in T} \tilde\epsilon_t (\mathbf{b}/\!\!/ \cdot)$
  and $\nu$ satisfy
  \begin{equation}\label{equ:e2}
    \dotprod\nolimits_{t \in T} \tilde\epsilon_t (\mathbf{b}/\!\!/  \cdot)
    \prec_\ell \dotprod\nolimits_{t \in T} \epsilon_t \qquad
    \text{and} \qquad 
    \underline\nu (\mathbf{1}) = 1. 
  \end{equation}
  By Lemma~\ref{lem:write-as-elem-prod} there exists an
  $\mathfrak{E}$-product $\dotprod_{s \in S} \tilde\epsilon_s$, with
  $\tilde\epsilon_s \in \gamma_{k+1}(\Omega)$ and
  $\deg(\tilde\epsilon_s) \ge 1$ for $s \in S$, such that
  \[
  \underline\nu( \mathbf{g}) = \prod\nolimits_{s\in S}
  \tilde\ue_s(\mathbf{g}) \qquad \text{for all
    $\mathbf{g} \in \mathbf{G}$.}
  \]
  Consider the $\mathfrak{E}$-product
  $\dotprod_{t \in T} \tilde\epsilon_t (\mathbf{b}/\!\!/ \cdot )
  \dotprod_{s \in S} \tilde\epsilon_s$,
  formally based on the ordinal sum of $T$ and~$S$, and set
  $\tilde \omega = \prod_{t \in T} \tilde\epsilon_t (\mathbf{b}/\!\!/
  \cdot ) \prod_{s \in S} \tilde\epsilon_s$.
  From~\eqref{equ:e2} and the fact that
  $\tilde\epsilon_s \in \gamma_{k+1}(\Omega)$ we deduce that
  \[
  \dotprod\nolimits_{t \in T} \tilde\epsilon_t (\mathbf{b}/\!\!/ \cdot
  ) \dotprod\nolimits_{s \in S} \tilde\epsilon_s \prec_\ell
  \dotprod\nolimits_{t \in T} \epsilon_t.
  \]
  Moreover, substituting $\mathbf{1}$ for $\mathbf{g}$
  in~\eqref{equ:e1}, we obtain
   \[
   \underline{\tilde \omega}(\mathbf{1}) = \prod\nolimits_{t \in T}
   \tilde\ue_t (\mathbf{b}/\!\!/ \mathbf{1}) \cdot \underline \nu(
   \mathbf{1}) = \left( \prod\nolimits_{t \in T} \ue_t (
     \mathbf{1})\right)^{-1} = \underline\omega(\mathbf{1})^{-1} = 1.
   \]
   Hence, by induction with respect to~$\preceq_\ell$, we conclude
   that
   $\dotprod_{t \in T} \tilde\epsilon_t (\mathbf{b}/ \cdot )
   \dotprod_{s \in S} \tilde\epsilon_s$
   is $\ell$-friendly and thus all the conditions in
   Definition~\ref{def:good-pw} are satisfied.
\end{proof}
 
\begin{lemma} \label{lem:identity-for-U} Let
  $\ell \colon \mathfrak{E} \to \mathsf{W}$ be a length function and
  let $\omega = \prod_{t \in T} \epsilon_t$, where
  $\dotprod_{t \in T} \epsilon_t$ is an $\ell$-friendly $r$-valent
  $\mathfrak{E}$-product for~$\mathbf{G}$.

  Suppose that $\mathbf{V} \le_\mathrm{c} \mathbf{G}$ is such that
  $\mathbf{V}_\omega = \{ \uo(\mathbf{v}) \mid
  \mathbf{v} \in \mathbf{V} \}$
  contains less than $2^{\aleph_0}$ elements.  Then there exists
  $\mathbf{U} \le_\mathrm{o} \mathbf{V}$ such that
  \[
  \uo(\mathbf{u}) = 1 \qquad \text{for all
    $\mathbf{u} \in  \mathbf{U} $.}
  \]
\end{lemma}

\begin{proof}
  We argue by induction, using the pre-order $\preceq_\ell$.  If
  $\dotprod_{t \in T} \epsilon_t$ is minimal with respect
  to~$\preceq_\ell$, the assertion holds for $\mathbf{U} = \mathbf{V}$, by
  Remark~\ref{rem:good-para-words}.

  Now suppose that $\dotprod_{t \in T} \epsilon_t$ is not minimal.  As
  $\lvert \mathbf{V}_\omega \rvert < 2^{\aleph_0}$,
  Proposition~\ref{pro:instead-of-Baire} implies that there are
  $\mathbf{b} \in \mathbf{V}$ and
  $\mathbf{U}_1 \le_\mathrm{o} \mathbf{V}$ such that
  $\uo(\cdot)$ is constant on the coset
  $\mathbf{b} \mathbf{U}_1$, i.e.\
  \[  
  \uo(\mathbf{b}) = \uo( \mathbf{b}
  \mathbf{u}) \qquad \text{for all $\mathbf{u} \in \mathbf{U}_1$}.
  \]

  By Definition~\ref{def:good-pw},  we obtain
  \[
  \uo( \mathbf{b} \mathbf{g}) =
  \uo(\mathbf{b}) \cdot \uo(\mathbf{g})
  \cdot \underline{\nu}(\mathbf{g}) \qquad \text{for all
    $\mathbf{g} \in \mathbf{G}$,}
  \]
  where $\nu = \prod_{s \in S} \widetilde{\epsilon}_s$ for an
  $\ell$-friendly $r$-valent $\mathfrak{E}$-product
  $\dotprod_{s \in S} \widetilde{\epsilon}_s$ for~$\mathbf{G}$ such
  that
  $\dotprod_{s \in S} \widetilde{\epsilon}_s \prec_\ell \dotprod_{t
    \in T} \epsilon_t$.  This yields
  \[
  \underline{\nu}(\mathbf{u}) = \uo(\mathbf{u})^{-1}
  \qquad \text{for all $\mathbf{u} \in \mathbf{U}_1$};
  \]
  in particular,
  $(\mathbf{U}_1)_\nu = \{ \underline{\nu}(\mathbf{u}) \mid \mathbf{u}
  \in \mathbf{U}_1 \}$
  has less than $2^{\aleph_0}$ elements.  By induction, we find the
  desired
  $\mathbf{U} \le_\mathrm{o} \mathbf{U}_1 \le_\mathrm{o} \mathbf{V}$
  such that
  \[
  \uo(\mathbf{u}) =
  \underline{\nu}(\mathbf{u})^{-1} = 1 \qquad \text{for all
    $\mathbf{u} \in \mathbf{U}$.}  \qedhere
  \]
\end{proof}

\begin{prop} \label{pro:good-words-OK} Let
  $\ell \colon \mathfrak{E} \to \mathsf{W}$ be a length function and
  let $\omega = \prod_{t \in T} \epsilon_t$, where
  $\dotprod_{t \in T} \epsilon_t$ is an $\ell$-friendly $r$-valent
  $\mathfrak{E}$-product for~$\mathbf{G}$.  Suppose that
  $\mathbf{V} \le_\mathrm{c} \mathbf{G}$ is such that
  $\mathbf{V}_\omega = \{ \uo(\mathbf{v}) \mid
  \mathbf{v} \in \mathbf{V} \}$
  has less than $2^{\aleph_0}$ elements.  Then $\mathbf{V}_\omega$ is
  already finite.
\end{prop}

\begin{proof}
  We argue by induction, using the pre-order $\preceq_\ell$.  If
  $\dotprod_{t \in T} \epsilon_t$ is minimal with respect
  to~$\preceq_\ell$, the assertion holds, by
  Remark~\ref{rem:good-para-words}: indeed,
  $\mathbf{V}_\omega = \{1\}$.
  
  Now suppose that $\dotprod_{t \in T} \epsilon_t$ is not minimal.  By
  Lemma~\ref{lem:identity-for-U}, there exists
  $\mathbf{U} \le_\mathrm{o} \mathbf{V}$ such that
  \[
  \uo(\mathbf{u}) = 1 \qquad \text{for all
    $\mathbf{u} \in \mathbf{U}$.}
  \]
  Let $B$ be any set of coset representatives for $\mathbf{U}$
  in~$\mathbf{V}$ so that
  $\lvert B \rvert = \lvert \mathbf{V} : \mathbf{U} \rvert < \infty$.

  By Definition~\ref{def:good-pw}, we see that, for each of the
  finitely many coset representatives $\mathbf{b} \in B$,
  \[
  \uo( \mathbf{b} \mathbf{u}) =
  \uo(\mathbf{b}) \cdot
  \underbrace{\uo(\mathbf{u})}_{=1} \cdot
  \underline{\nu}_{\mathbf{b}}(\mathbf{u}) =
  \underbrace{\uo(\mathbf{b})}_{\text{constant}}
  \underline{\nu}_{\mathbf{b}}(\mathbf{u}) \qquad \text{for
    $\mathbf{u} \in \mathbf{U}$,}
  \]
  where
  $\nu_{\mathbf{b}} = \prod_{s \in S(\mathbf{b})}
  \widetilde{\epsilon}_{\mathbf{b},s}$
  for an $\ell$-friendly $r$-valent $\mathfrak{E}$-product
  $\dotprod_{s \in S(\mathbf{b})} \widetilde{\epsilon}_{\mathbf{b},s}$
  for~$\mathbf{G}$ such that
  $\dotprod_{s \in S(\mathbf{b})} \widetilde{\epsilon}_{\mathbf{b},s}
  \prec_\ell \dotprod_{t \in T} \epsilon_t$.
  In particular, for each $\mathbf{b} \in B$ the set
  \[
  \mathbf{U}_{\nu_\mathbf{b}} = \{
  \underline{\nu}_\mathbf{b}(\mathbf{u}) \mid \mathbf{u} \in \mathbf{U}
  \} = \{ \uo(\mathbf{b})^{-1} \uo(
  \mathbf{b} \mathbf{u}) \mid \mathbf{u} \in \mathbf{U} \} \subseteq
  \uo(\mathbf{b})^{-1} \mathbf{V}_\omega
  \]
  has less than $2^{\aleph_0}$ elements.  By induction, each
  $\mathbf{U}_{\nu_\mathbf{b}}$ is finite, hence also the finite union
  \[
  \mathbf{V}_\omega = \bigcup_{\mathbf{b} \in B} \uo(\mathbf{b})
  \mathbf{U}_{\nu_\mathbf{b}}. \qedhere
  \]
\end{proof}

%%%%%

\section{Nilpotent groups and specific words} \label{sec:nilp-groups}

In this section we prove Theorem~\ref{thm:w-str-concise-in-nilp-gp} and its corollaries.

\begin{lemma} \label{lem:red-to-comm} Let $\mathcal{C}$ be a class of
  profinite groups such that every commutator word is strongly concise
  in~$\mathcal{C}$.  Then every word is strongly concise
  in~$\mathcal{C}$.
\end{lemma}
          
\begin{proof}
  Let $w = w(x_1,\ldots,x_{\kr}) \in F$, where
  $F = \langle x_1, \ldots, x_{\kr} \rangle$ is a free group of
  rank~$\kr \ge 2$.  Write $w = u v$, where
  $u = x_1^{\, e_1} \cdots x_{\kr}^{\, e_{\kr}}$ with
  $e_1, \ldots, e_{\kr} \in \mathbb{Z}$ and where $v = v(x_1,\ldots,x_{\kr})$
  is a commutator word, i.e.\ $v \in F'$.

  Suppose that $\lvert G_w \rvert <2^{\aleph_0}$.  Let
  $m=\gcd(e_1,\ldots,e_{\kr})$ and choose
  $f_1, \dots , f_{\kr} \in \mathbb{Z}$ such that
  $m = \sum_{i=1}^{\kr} e_i f_i$.  We observe that
  $g^m = w(g^{f_1},\ldots,g^{f_{\kr}})$, for every $g \in G$.  Thus
  $\{ g^m \mid g \in G \}$ has less than $2^{\aleph_0}$ elements, and
  consequently $G_u$ has less than $2^{\aleph_0}$ elements, as
  $m = \gcd(e_1,\ldots,e_{\kr})$. Therefore, also $G_v$ has less than
  $2^{\aleph_0}$ elements.  Since $v$ is strongly concise in~$G$, the
  group $v(G)$ is finite.  Working modulo $v(G)$, we may assume that
  $v(G) = 1$ and $w=u$.  To simplify the notation, we may further
  assume that $w = x_1^{\, m}$.

  As $G_w$ has less than $2^{\aleph_0}$ elements, so does
  $G_{\widetilde{v}}$ for the commutator word
  $\widetilde{v} = \widetilde{v} (x_1,x_2) = (x_1 x_2)^{-m} x_1^{\, m}
  x_2^{\, m}$.
  Since $\widetilde{v}$ is strongly concise in~$G$, the group
  $\widetilde{v}(G)$ is finite.  Working modulo $\widetilde{v}(G)$, we
  may assume that $\widetilde{v}(G) = 1$ and thus
  \[
  (gh)^m = g^m h^m \qquad \text{for all $g,h \in G$.} 
  \]
  Hence $G_w = w(G)$ is the image of the homomorphism $G \to G$,
  $g \mapsto g^m$.  From $\lvert w(G) \rvert < 2^{\aleph_0}$ we
  conclude that $w(G)$ is finite.
\end{proof}

\begin{proof}[Proof of Theorem~{\rm\ref{thm:w-str-concise-in-nilp-gp}}]
  Let $w$ be a group word, and let $G$ be a nilpotent profinite group
  of class~$c$.  Suppose that $\lvert G_w \rvert < 2^{\aleph_0}$.
  We claim that $w(G)$ is finite.  By Lemma~\ref{lem:red-to-comm}, we
  may suppose that $w$ is a commutator word.

  Clearly, $G = \prod_p H_p$, where the product runs over all
  primes~$p$ and $H_p$ denotes the unique Sylow pro-$p$ subgroup
  of~$G$.  We conclude that $G_w = \prod_p (H_p)_w$ and
  $w(G) = \prod_p w(H_p)$.  As $\lvert G_w \rvert < 2^{\aleph_0}$,
  this implies $w(H_p) = 1$ for all but finitely many primes~$p$.

  Consequently, we may suppose that $G$ is a pro-$p$ group.  As $G$ is
  nilpotent, it satisfies the hypothesis of
  Corollary~\ref{cor:reduction-to-pro-p-central-exp-p} (see
  \cite[Theorem~4.1.5]{Se09}).  Hence we may further suppose that
  $w(G)$ is central and of exponent~$p$.  Using
  Lemma 
  ~\ref{rem:gamma_c}, we apply Proposition~\ref{pro:good-words-OK} to
  deduce that $G_w$ is finite.  Thus $w(G)$ is a finitely generated
  elementary abelian group and therefore finite.
\end{proof}

\begin{proof}[Proof of Corollary~{\rm\ref{cor:F-mod-w(F)-nilp-OK}}]
 Let $G$ be a profinite group such that
  $\lvert G_w \rvert < 2^{\aleph_0}$. 
  Suppose that $F/w(F)$ has nilpotency class~$c$. Then
  $\gamma_{c+1} = [x_1,\ldots, x_{c+1}]$ can be written as the product
  of finitely many $w$-values or their inverses.  Hence $\gamma_{c+1}$
  has less than $2^{\aleph_0}$ values in~$G$, and
  Theorem~\ref{thm:mcw-more} implies that $\gamma_{c+1}(G)$ is finite.
  Passing to the quotient $G/\gamma_{c+1}(G)$, we can assume that $G$
  nilpotent and Theorem~\ref{thm:w-str-concise-in-nilp-gp} applies.
\end{proof}

\begin{lemma} \label{lem:red-to-prime-powers}
  Let $m \in \mathbb{N}$ with prime factorisation
  $m = p_1^{\, e_1} \cdots p_k^{\, e_k}$.  Suppose that, for each
  $i \in \{1, \ldots, k\}$, the word $x^{p_i^{e_i}}$ is strongly
  concise in the class of pro-$p_i$ groups.  Then the word $x^m$ is
  strongly concise in the class of all profinite groups.
\end{lemma}

\begin{proof}
  Put $w = x^m$ and let $G$ be a profinite group such that
  $\lvert G_w \rvert < 2^{\aleph_0}$.  By
  Lemma~\ref{lem:not-comm-periodic}, the group $G$ is periodic.  A
  theorem of Herfort~\cite{He79} yields that the group $G$ has
  non-trivial Sylow pro-$p$ subgroups for only finitely many
  primes~$p$.
  
  Suppose that $q$ is a prime not dividing $m$, and let $Q$ be a Sylow
  pro-$q$ subgroup of~$G$.  Then every element of $Q$ is an $m$th
  power.  Consequently, $Q$ is finite and each of its elements has
  only finitely many conjugates in~$G$, by
  Lemma~\ref{lem:fin-conj-cl}.  Hence $Q$ is contained in a finite
  normal subset of $G$ consisting of elements of finite order.  By
  Dicman's Lemma~\cite[14.5.7]{Ro96}, the group $Q$ is contained in a
  finite normal subgroup of~$G$.
  
  Consequently, there exists a finite normal subgroup
  $N \trianglelefteq_\mathrm{c} G$ that contains all Sylow pro-$q$
  subgroups of~$G$, for primes $q$ not dividing~$m$.  Passing to
  $G/N$, we may suppose that $G$ is a pro-$\{p_1,\ldots,p_k\}$ group.
  Fix $i \in \{1,\ldots,k\}$ and let $P_i$ be a Sylow pro-$p_i$
  subgroup of $G$.  Then the set of $p_i^{\, e_i}$th powers in $P_i$
  is the same as the set of $m$th powers.  Thus the set has less than
  $2^{\aleph_0}$ elements and our assumptions yield that the group
  $K_i = P_i^{\, p_i^{e_i}} = \langle g^{p_i^{e_i}} \mid g \in P_i
  \rangle$
  is finite.  Each element of $K_i$ has only finitely many conjugates
  in~$G$; compare Lemma~\ref{lem:fin-conj-cl}.  Thus $K_i$ is
  contained in a finite normal subgroup $N_i \trianglelefteq G$, again
  by Dicman's Lemma.

  Factoring out the finite normal subgroup $N_1 N_2 \cdots N_k$, we
  may suppose that each of the Sylow pro-$p_i$ subgroups $P_i$ has
  exponent dividing $p_i^{\, e_i}$.  Thus $G$ has exponent dividing
  $m = p_1^{\, e_1} \cdots p_k^{\, e_k}$ and $w(G) =1$
\end{proof}

\begin{proof}[Proof of Corollary~{\rm\ref{cor:specific-words}}]
  By Lemma~\ref{lem:red-to-prime-powers}, the assertion for $x^6$
  follows once we have dealt with the words $x^2$ and $x^3$.  Let $F$
  be a free group of countably infinite rank, and let $w$ be one of
  the specific words, other than $x^6$, that appear in the statement
  of the corollary.  It suffices to show that $F/w(F)$ is nilpotent,
  so Corollary~\ref{cor:F-mod-w(F)-nilp-OK} can be applied.

  \begin{enumerate}
  \item[(i)] $w = x^2$.  It is well known that $F/w(F)$ s abelian.
  \item[(ii)] $w = [x^2,z_1,\ldots,z_r]$.  Extending the argument
    given in (i), we see that $F/w(F)$ is nilpotent of class at
    most~$1+r$.
  \item[(iii)] $w = x^3$.  Since $F/w(F)$ has exponent~$3$, it is a
    $2$-Engel group and thus nilpotent of class at most~$3$, by a
    classical result of Hopkins~\cite{Ho29};
    compare~\cite[12.3.6]{Ro96}.
  \item[(iv)] $w = [x^3,z_1,\ldots,z_r]$ for $r \ge 1$.  Extending the
    argument given in (iii), we see that $F/w(F)$ is nilpotent of
    class at most~$3+r$.
  \item[(v)] $w = [x,y,y]$.  Every $2$-Engel group is nilpotent of
    class at most~$3$.
  \item[(vi)] $w = [x,y,y,z_1,\ldots,z_r]$ for $r \ge 1$. Extending
    the argument given in (v), we see that $F/w(F)$ is nilpotent of
    class at most~$3+r$.  %\qedhere
  \end{enumerate}
\end{proof}

\begin{acknowledgements}
  In connection with our original proof of Theorem~\ref{thm:mcw-more}
  we acknowledge discussions with Marta Morigi; subsequent comments of
  the referee led to a significant simplification of our argument.  We
  also acknowledge the referee's general feedback which led to several
  improvements of the presentation.
\end{acknowledgements}

%%%%%

\end{document}